\documentclass[preprint,12pt,3p]{elsarticle}
\usepackage{amsmath,amsthm,amsfonts,amssymb}
\usepackage{subfig}
\usepackage{blkarray}
\usepackage{caption}
\usepackage{float}
\usepackage{algpseudocode}
\usepackage{tikz}
\usepackage{hyperref}

\numberwithin{equation}{section}
\usepackage{xcolor}
\usepackage{colortbl}
\usepackage[normalem]{ulem}
\usepackage{sectsty}
\usepackage{colortbl}

\definecolor{warmblack}{rgb}{0.0, 0.26, 0.26}
\usepackage{multirow}
\usepackage{mathtools}
\usepackage{algorithm}
\usepackage{algorithmicx}
\makeatletter
\def\BState{\State\hskip-\ALG@thistlm}
\makeatother
\newtheorem{theorem}{Theorem}[section]
\newtheorem{lemma}[theorem]{Lemma}
\theoremstyle{definition}
\usepackage{hyperref}
\newtheorem{definition}{Definition}[section]

\newtheorem{example}{Example}[section]
\journal{FILOMAT}
\def\rnn{{\mathbb R}^{n\times n}}
\def\rmn{{\mathbb R}^{m\times n}}

\begin{document}

\begin{frontmatter}

\title{ \textcolor{warmblack}{\bf Regularized Iterative Method for Ill-posed Linear Systems Based on Matrix Splitting}}

\author[label1]{Ashish Kumar Nandi}
\address[label1]{Department of Mathematics, BITS Pilani, K.K. Birla Goa Campus, Goa, India}
\ead{ashish.nandi123@gmail.com}
\cortext[cor1]{Corresponding author}

\author[label1]{Jajati Keshari Sahoo\corref{cor1}}
%\address[label2]{Student, BITS Pilani, K.K. Birla Goa Campus, Goa, India}
\ead{jksahoo@goa.bits-pilani.ac.in}

%\author[label2]{Pushpendu Ghosh}
%\ead{f20150366@goa.bits-pilani.ac.in}

\begin{abstract}
In this paper, the concept of matrix splitting is introduced to solve a large sparse ill-posed linear system via Tikhonov’s regularization. In the regularization process, we convert the ill-posed system to a well-posed system. The convergence of such a well-posed system is discussed by using different types of matrix splittings. Comparison analysis of both systems are studied by operating certain types of weak splittings. Further, we have extended the  double splitting of [Song J. and Song Y, Calcolo 48(3), 245–260, 2011]  to double weak splitting of type II for nonsingular symmetric matrices. In addition to that, some more comparison results are presented with the help of such weak double splittings of type I and type II. 
\end{abstract}

\end{frontmatter}

\section{Introduction}
In the view of Hadamard \cite{lahadaj1}, the discretization of  Fredholm integral equations of the first kind \cite{lagrcw1} is  formed an ill-posed linear system
\begin{equation}\label{laeq1}
Ax = b, 
\end{equation}
where $A\in\rmn,~{x\in{\mathbb{R}^{n}}}$  and ${b\in{\mathbb{R}^{m}}}$. In practice, this type of ill-posed system appears in several branches of science and engineering such as noisy image restoration \cite{land1}, computer tomography \cite{laher1} and inverse problems within electromagnetic \cite{larom1}. Ill-posed problems were extensively studied in the context of an inverse problem \cite{lacalrz1,engl, groet} and image restorations \cite{bert}. In image restoration, the main objective is to establish a blurred free image that requires the approximate solution of the system (\ref{laeq1}). For more details one can refer \cite{land1,bert}. To find the approximate solution of the ill-posed system (\ref{laeq1}), several iterative methods such as Accelerated Landweber iterative method \cite{lahanke1}, GMRES and singular preconditioner method \cite{laeldsv1}, conjugate gradient method are studied in the recent past. However, the utilization of the splittings method along with regularization is quite a new idea. 

In order to solve the system $Ax=b$, i.e, find the least square solution $A^\dagger b$, we first normalize as $A^{T}Ax = A^{T}b$. This does not make the problem simple as most of the cases the matrix $A^{T}A$ is singular and ill-conditioned which is affected highly by round-off errors \cite{laggn1}. Thus we need to make the system $Ax=b$,  well-posed by introducing a regularization parameter $\lambda (>0)$, and the corresponding modified well-posed system based on Tikhonov's regularization \cite{latikon1} is given by
\begin{equation}\label{laeq2}
(A^{T}A+\lambda I)x = A^{T}b. 
\end{equation}
If we consider $B_{\lambda} = A^{T}A+\lambda I$, then the system $(\ref{laeq2})$ reduces to the following system 
\begin{equation}\label{lareq1}
   B_{\lambda}x =  A^{T}b. 
\end{equation}
The above procedure is known as regularization and the parameter $\lambda$ determines what extent the original ill-posed system (\ref{laeq1}) is changed. There are several ways to regularize such type of ill-posed system. Among them, the most classical regularization is {\it Tikhonov's regularization} introduced by Tikhonov in $1963$ \cite{latikon1}. Some iterative methods for the system (\ref{laeq2}) in framework of operator theory can be found in \cite{lahgcw1} and the references therein. The main motivation to analyze and compare the numerical solution of both system (\ref{laeq1}) and (\ref{lareq1}) is comes from Barata and Hussein \cite{labarh1}, where the authors have shown that $B_{\lambda}^{-1}A^Tb\to A^\dagger b$ as $\lambda\to 0$.

On the other hand, the matrix splitting (A decomposition $A=U-V$ is called a splitting of the matrix $A$) methods are more significant and numerically stable in dealing with rectangular matrices.  In this direction, Berman and Plemmons \cite{labpc1} first introduced the \textit{proper splitting} ( A splitting  $A = U-V$  is called \textit{proper} if the null space of $A$ is equal to the null space of $U$ and the range space of $A$ is equal to range space $U$). If $A=U-V$ is a proper splitting of $A\in\rmn$, then the associated iterative scheme for solving $Ax=b$,  is given by
\begin{equation}\label{laaeq4}
x^{k+1} = U^{\dag}Vx^{k}+U^{\dag}b.
\end{equation}
It is well known that the iterative scheme defined in  (\ref{laaeq4}) converges to $A^{\dag}b$ if and only if the spectral radius of $U^\dagger V$ is less than $1$. Further, if the system $Ax=b$ is consistent, then the
above iterative process converges to a solution of \eqref{laeq1}. In \cite{labpc1}, it was proved
that if $A = U-V$ is a proper splitting such that $U^\dagger\geq 0$ (entry-wise) and $U^\dagger V\geq 0$, then $A^\dagger\geq 0$ if and only if the spectral radius of $U^\dagger V$ is less than $1$. 

In case of nonsingular coefficient matrix $B_{\lambda}$, if $B_{\lambda}=M_{\lambda}-N_{\lambda}$ is a splitting of $B_{\lambda}\in\rnn$ such that $M_{\lambda}$ is invertible, then the associated iterative is given by \begin{equation}\label{eq113}
 x^{k+1} = M_{\lambda}^{-1}N_{\lambda}x^{k}+M_{\lambda}^{-1}A^{T}b. 
\end{equation}
It is clear that this iterative method converges to $B_{\lambda}A^Tb~(=A^\dagger B as \lambda \to 0)$ if and
only if the spectral radius of $M_{\lambda}^{-1}N_{\lambda}$  is less than $1$.  We call a splitting convergent if the associated iterative scheme convergent. Several types of splittings and numerous comparative studies can be found in the literature \cite{lalic2014,  lamsz1, lamiao2012,lamsd1,lason1,lawzi1} and the reference therein.

The main objective of this article is to introduce a new regularized nonsingular approach for the rectangular or singular system and study the convergence of the iterative method \ref{eq113} associated with different types of splittings of $B_{\lambda}$. In case of the regularized iterative scheme, we can relax some strong conditions such as non-negativeness of $A$, $A^{\dagger}$ to assure the convergence. Besides, we have introduced a new matrix splitting, called double weak splitting of type II. Further, several comparative studies between the original system \eqref{laeq1} and regularized system \eqref{lareq1} are provided. The theoretical results provided shows that the regularized iterative scheme convergence faster (in terms of spectral radius).

\subsection{Outline}
The paper is organized as follows: Some useful notations and definitions are discussed in  Section \ref{prel}. In addition to these we review some basic theories of iterative methods, which will be used throughout this paper. The main results of this article is elaborated in Section \ref{main}. Numerous comparison results related to the systems (\ref{laeq1}) and (\ref{lareq1}) are established. Also, a double weak splitting of type II newly introduced as well as a few comparison theorems have been proved for the double weak splitting of type II. The manuscript is concluded along with a few future research perspectives in Section \ref{con}.

\section{Preliminaries}\label{prel}
First, we elaborate on some notations and definitions which will be useful throughout the article. The set of all real rectangular matrices of order $m\times n$ is denoted by $\mathbb{R}^{m\times{n}}$. For matrices $A,B\in\mathbb{R}^{m\times{n}}$, a matrix  $B$ is said to be {\it nonnegative} ($B\geq 0$) if all entries of $B$ are nonnegative and $A\geq B$ implies $A-B\geq 0$. If $L$ and $M$ are two complementary subspaces of $\mathbb{R}^{n}$, then $P_{L,M}$ is the projection on $L$ along $M$. So, $P_{L,M}B = B$ if and only if $R(B)\subseteq L$ and $BP_{L,M} = B$ if and only if $N(B)\supseteq M$. Henceforth, $R(A)$ and $N(A)$ denotes the range space and null space of the matrix $A$. We denote the transpose of a matrix $A$  by $A^{T}$.  The {\it spectral radius} of a matrix $B\in\mathbb{R}^{n\times{n}}$ is denoted as $\rho(B)$ and defined by $\rho(B) = \max\limits_{1\leq i\leq n} |\sigma_{i}|$, where $\sigma_{i}$'s are the eigenvalues of $B$. It is well known that for any square matrix $B$, $\rho(B^{T}) = \rho(B)$ and $\rho(AB) = \rho(BA)$ for well defined product of matrices $A,B$. We recall the Moore-Penrose inverse of a matrix $B$. The unique
matrix $X\in\mathbb{R}^{n\times{m}}$, satisfying $BXB = B,~XBX=X,~(BX)^{T}=BX~\text{and}~ (XB)^{T}=XB$, is called the {\it Moore-Penrose inverse} of $B$ and denoted by $B^\dagger$. A few properties of $B^\dagger$ which are frequently being used: $R(B^{T}) = R(B^{\dag})$; $N(B^{T}) = N(B^{\dag})$; $B^{\dag}B = P_{R(B^{T})}$ and $BB^{\dag} = P_{R(B)}$.  Further, a non singular matrix $B$ is called {\it monotone} if $A^{-1}\geq 0$. Similarly, we cal a matrix $B\in\mathbb{R}^{m\times{n}}$  {\it semi-monotone} if $B^{\dag}\geq 0$. 

Next, we discuss some necessary results based on non-negativeness regularization, matrix splittings. The very first result is  for nonnegative matrices. 
\begin{theorem}\cite{labpn1}\label{2.4}
Let  $B\in\rnn$, $B\geq 0$, $x\geq 0$ $(x\neq 0)$ and $\alpha$ be a positive scalar. Then the following are holds.
\begin{enumerate}
    \item[(i)] If $\alpha x\leq Bx$, then $\alpha\leq \rho(B)$.
    \item[(ii)] For $x>0$, if $Bx\leq \alpha x$,  then $\rho(B) \leq \alpha$.
\end{enumerate}
\end{theorem}
We now collect a few parts of the classical Perron-Frobenius theorem. Perron proved it for positive matrices and Frobenius gave the extension to irreducible matrices. 
\begin{theorem}[Theorem $2.20$, \cite{lavm1}]\label{2.1}
Let $A\in\rnn$ be a nonnegative matrix. Then
\begin{enumerate}
    \item[(i)] $A$ has a nonnegative real eigenvalue equal to its spectral radius.
        \item[(ii)] $(ii)$ there exists a nonnegative eigenvector for its spectral radius.
        \end{enumerate}
\end{theorem}
\begin{theorem}[Theorem $2.7$, \cite{lavm1}]\label{la2.10}
Let $A\in\rnn$ be a nonnegative matrix. If $A$ is irreducible, then 
\begin{enumerate}
    \item[(i)] $A$ has a positive real eigenvalue equal to its spectral radius.
\item[(ii)] there exists a positive eigenvector for its spectral radius.
\end{enumerate}
\end{theorem}
In connection to the spectral radius, the following result is collected from \cite{lavm1}. 
\begin{theorem}[Theorem $2.21$, \cite{lavm1}]\label{2.25}
If $A, B\in\rnn$ and $A\geq B\geq 0$, then $\rho(A)\geq \rho(B)$.
\end{theorem}

In view of proper splitting, we state the following essential results \begin{theorem}[Theorem $1$, \cite{lacdp1}]\label{2.8}
Let $A = U-V$ be a proper splitting of $A\in\rmn$. Then
\begin{enumerate}
    \item[(i)]  $A = (I-VU^{\dag})U = U(I-U^\dagger V)$,
 \item[(ii)] $I-VU^{\dag}$ is nonsingular,
 \item[(iii)] $A^{\dag} = U^{\dag}{(I-VU^{\dag})}^{-1}=(I-U^\dagger V)^{-1}U^\dagger$.
\end{enumerate}
\end{theorem}
\begin{theorem}[Theorem $2.2$, \cite{lams1}]\label{2.6}
Let $A = U-V$ be a proper splitting of $A\in\rmn$. Then
\begin{enumerate}
    \item[(i)]  $UU^{\dag} = AA^{\dag}$ and $U^{\dag}U = A^{\dag}A$,
\item[(ii)] $U^{\dag} = (I+A^{\dag}V)^{-1}A^{\dag} = A^{\dag}(I+VA^{\dag})^{-1}$,
\item[(iii)] $U^{\dag}VA^{\dag} = A^{\dag}VU^{\dag}$.
\end{enumerate}
\end{theorem}

Next, we recall the definition of weak proper splitting of the first type and the second type 
\begin{definition}[Definition $2$, \cite{lacdp1}]\label{3.1.1}
A proper splitting $A = U-V$ of $A\in\rmn$ is called a weak proper splitting of the first type (respectively, the second type), if $U^{\dag}V\geq 0$ (respectively, $VU^{\dag}\geq 0$).
\end{definition}
In case of nonsingular matrices, the splittings are defined in Definition \ref{3.1.1}, are called respectively as weak splitting of the first type and weak splitting of the second type (which were respectively  introduced by Marek $\&$ Szyld \cite{lamsz1}), and by Wo{\'z}nicki \cite{lawz1}) and stated in the next definition.

\begin{definition}\label{3.1.13}
A splitting $A = U-V$ of $A\in\rnn$ is called a weak splitting of the first type, if $U^{-1}V \geq 0$ and a weak splitting of the second type, if $VU^{-1} \geq 0$.
\end{definition}
The next result is a combination of Theorem $2$ and Remark $2$ of \cite{lacdp1}.
\begin{theorem}\label{3.1.3}
Let $A = U-V$ be a weak proper splitting of the first type (or second type) of $A\in\rmn$. Then $A^{\dag}V~(\mbox{or}~VA^{\dag})\geq 0$ if and only if $\rho(U^{\dag}V) = \frac{\rho(A^{\dag}V)}{1+\rho(A^{\dag}V)}<1$ (respectively $\rho(VU^{\dag}) = \frac{\rho(VA^{\dag})}{1+\rho(VA^{\dag})}<1$).
\end{theorem}
In a special case of the above result (Theorem \ref{3.1.3}), which was proved in \cite{laclip1} is stated in the next theorem.  
\begin{theorem}[Theorem $3$ and Remark $4$ of \cite{laclip1}]\label{3.1.5}
Let $A\in\rnn$ be nonsingular, and let $A = U-V$ be a weak splitting of the first type (respectively, the second type). Then $A^{-1}V \geq 0$ (respectively, $VA^{-1} \geq 0$ ) if and only if $\rho(U^{-1}V) = \frac{\rho(A^{-1}V)}{1+\rho(A^{-1}V)}<1$ (respectively, $\rho(VU^{-1}) = \frac{\rho(VA^{-1})}{1+\rho(VA^{-1})}<1$).
\end{theorem}

Further, we recall one comparison theorem of  \cite{lagiri2017} for two weak splittings of the second type.

\begin{theorem}[Corollary $3.13$, \cite{lagiri2017}]\label{giricor}
Let $A=M_1-N_1=M_2-N_2$ be two weak splittings of the second type of $A\in\rnn$ with $M_{i}A^{-1}\geq 0$, $i=1,2$. If there exist index $j\geq 1$ and $\alpha~(0<\alpha<1)$, such
that $(M_1A^{-1})^j\leq \alpha(M_2A^{-1})^j$, then $\rho(N_1M_1^{-1})<\rho(N_2M_2^{-1})$.
\end{theorem}

The notion of double splitting was first introduced by Wo{\'z}nicki \cite{lawoz1} in $1993$. Later, several characterizations of double splitting were investigated by many researchers (once can refer \cite{lamio2009}, \cite{lasheh1}, and \cite{lashts07}). In addition to these, Song and Song \cite{lasons1} introduced the double nonnegative splitting to discuss the iterative solution of the nonsingular system $Ax = b$. Further, the comparison results of  \cite{lasons1} have been  extended by the authors of \cite{lalic2014}, \cite{lalicx2014}, and \cite{lamiao2012}. For convenience, we have renamed the double nonnegative splitting as the double weak splitting of type I. Hence the Definition $1.3$ of \cite{lasons1} is restated as follows.
\begin{definition}\label{w1.1}
The splitting $A = P-R+S$ is called double weak splitting of type I of  a nonsingular matrix $A\in\mathbb{R}^{n\times{n}}$ if $P^{-1}R\geq 0$ and $-P^{-1}S\geq 0$.
\end{definition}

If $A= P-R+S$  be a double weak splitting of type I of a nonsingular matrix $A\in\rnn$, then the iterative solution to the $Ax = b$, can be easily obtained by the following iterative scheme
\begin{equation*}
x^{k+1} = P^{-1}R x^{k}-P^{-1}S x^{k-1}+P^{-1}b.   
\end{equation*}
Further, its block matrix representation is given by 
\begin{equation}\label{eq22.1}
\begin{pmatrix}
    x^{k+1} \\
    x^{k}
  \end{pmatrix}
  = \begin{pmatrix}
    P^{-1}R & -P^{-1}S \\
    I & 0
  \end{pmatrix} \begin{pmatrix}
    x^{k} \\
    x^{k-1}
  \end{pmatrix}
  +  \begin{pmatrix}
    P^{-1}b \\
    0
  \end{pmatrix}=  \widehat{\mathbf{W}}\begin{pmatrix}
    x^{k} \\
    x^{k-1}
  \end{pmatrix}
  +  \begin{pmatrix}
    P^{-1}b \\
    0
  \end{pmatrix},
\end{equation}
where $I$ is an identity matrix of order $n$ and the iteration matrix $\widehat{\mathbf{W}}$ is, 
\begin{equation*}\label{wc2.1}
     \widehat{\mathbf{W}} = \begin{pmatrix}
    P^{-1}R & -P^{-1}S \\
    I & 0 
    \end{pmatrix}.
\end{equation*}
The convergence of the above iterative scheme which was proved by Song  and Song \cite{lasons1}, is given in the following theorem.

\begin{theorem}[\cite{lasons1}]\label{d2.21}
Let $A= P-R+S$  be a double weak splitting of type I of a nonsingular matrix $A\in\rnn$. Then the iterative scheme defined in equation \eqref{eq22.1} converges to $A^{-1}b$ if and only if $\rho(\widehat{\mathbf{W}})<1$.
\end{theorem}
Further, an equivalent characterization for a double weak splitting of type I, is stated  below. 

\begin{theorem}[Theorem $2.4$, \cite{lasons1}]\label{weaktp1.2}
Let $A= P-R+S$  be a double weak splitting of type I of a nonsingular matrix $A\in\rnn$. Then the following conditions are equivalent:
\begin{enumerate}
    \item[(i)] $\rho(\widehat{\mathbf{W}})<1$.
    \item[(ii)] $\rho(P^{-1}(R-S))<1$.
     \item[(iii)] $A^{-1}P\geq 0$.
      \item[(iv)] $A^{-1}P\geq I$.
\end{enumerate}
\end{theorem}
 Followed by the remarkable work of Neumann \cite{laneum1}, Jena et al. \cite{lajmp1} introduced double proper splitting as follows.
 
A decomposition $A = P-R+S$ of $A\in\rmn$ is called \textit{double proper splitting} if $R(A) = R(P)$ and $N(A) = N(P)$.

Applying the double proper splitting $A = P-R+S$ to the system (\ref{laeq1}), we get the following iterative scheme. 

\begin{equation*}
 x^{k+1} = P^{\dag}Rx^{k}-P^{\dag}Sx^{k-1}+P^{\dag}b, ~~k>0.
 \end{equation*}

Further, its block matrix form is given by 
\begin{equation}\label{eq7}
\begin{pmatrix}
    x^{k+1} \\
    x^{k}
  \end{pmatrix}
  = \begin{pmatrix}
    P^{\dag}R & -P^{\dag}S \\
    I & 0
  \end{pmatrix} \begin{pmatrix}
    x^{k} \\
    x^{k-1}
  \end{pmatrix}
  +  \begin{pmatrix}
    P^{\dag}b \\
    0
  \end{pmatrix}.
\end{equation}
The authors of \cite{lajmp1} have proved that the iterative scheme (\ref{eq7}) converges to the unique least square solution $A^{\dag}b$ of (\ref{laeq1}) if the spectral radius of the iteration matrix
\begin{equation}\label{eq8}
 \mathbf{W} = \begin{pmatrix}
    P^{\dag}R & -P^{\dag}S \\
    I & 0   
\end{pmatrix}
\end{equation}
is less than one, i.e., $\rho(\mathbf{W})<1$. More on the convergence of the scheme (\ref{eq7}) concerning different types of splittings and its comparison analysis can be found in \cite{lajmp1},\cite{ladm1}, and  \cite{lawang14}. In addition to these, Mishra \cite{ladm1} introduced the double proper nonnegative splitting which we renamed as the double proper  weak splitting of type I and defined as follows.

\begin{definition}\label{2.3}
A decomposition $A = P-R+S$ is called a double proper weak splitting of type I if $R(A) = R(P)$, $N(A) = N(P)$, $P^{\dag}R\geq 0$ and $P^{\dag}S\leq 0$.
\end{definition}

The convergence of double proper weak splitting have proved by Mishra  \cite{ladm1} stated below.
\begin{theorem}[Theorem $4.5$, \cite{ladm1}]\label{2.1.10}
Let $A^{\dag}P\geq 0$. If $A = P-R+S$ is a double proper  weak splitting of type I of $A\in\rmn$, then $\rho(\mathbf{W})<1.$
\end{theorem}

At the end of this section, we collect a few results based on the existence and the convergence of regularized splitting. 
\begin{theorem}[Lemma $4.2$, \cite{labarh1}]\label{2.1.7}
For all $A\in\rmn$, $$\lim_{\lambda\to 0} (A^{T}A+\lambda I)^{-1}A^{T} = \lim_{\lambda\to 0}B_{\lambda}^{-1}A^{T} \mbox{ exists}.$$ 
\end{theorem}

\begin{theorem}[Theorem $4.3$, \cite{labarh1}]\label{2.1.8}
For all $A\in\rmn$,
$$\lim_{\lambda\to 0} (A^{T}A+\lambda I)^{-1}A^{T} = A^{\dag} = \lim_{\lambda\to 0}B_{\lambda}^{-1}A^{T}.$$
\end{theorem}

\section{Main Results}\label{main}
This section has three parts. In the first part of this section, we discuss some convergence and comparison results related to the weak splitting of the first type and second type. The concept of double weak splittings of type II is introduced in the second part. In addition, several results based on double weak splittings of type II has been discussed. In the last part, we study for double proper weak splitting and its comparison with respect to the double weak splitting of type II.
\subsection{Convergence \& comparison using weak splittings}\label{sub3.1}
We first study the convergence of regularized iterative scheme (\ref{eq113}) for the well-posed system (\ref{lareq1}).
In view of  Theorem \ref{2.1.8} and Theorem \ref{3.1.5}, it is clear that the iterative scheme (\ref{eq113}) converges to $A^\dagger b$ and summarized in the next result.
 
\begin{theorem}\label{4.1conv}
Let $A\in\rmn$. For $\lambda>0$, if $B_{\lambda} = M_{\lambda}-N_{\lambda}$ is a  weak splitting of the first type (respectively, second type) of $B_\lambda\in\rnn$ with $\lim_{\lambda\to 0}B_{\lambda}^{-1}N_{\lambda}\geq 0$ (respectively, $\lim_{\lambda\to 0}N_{\lambda}B_{\lambda}^{-1}\geq 0$), then the iterative scheme $(\ref{eq113})$ converges to $B_{\lambda}^{-1}A^{T}b= A^{\dag}b$ as $\lambda \to 0$.
\end{theorem}
Due to the fact that both system (\ref{laeq1}) and (\ref{lareq1}) convergence to the same least square solution  $A^{\dag}b$, it is better to study and analyze the spectral radius of the respective iteration matrix. Motivated by Theorem $3.11$ of \cite{labals2015},  we have an affirmative answer to these spectral raddi and stated below.

\begin{theorem}\label{pp4.1conv}
Let $A = M-N$ be a weak proper splitting of the first type of $A\in\rmn$. For  $\lambda >0$, let $B_{\lambda} = M_{\lambda}-N_{\lambda}$ be a weak splitting of the first type of  $B_{\lambda}\in\rnn$. If $A^{\dag}N\geq \lim_{\lambda\to 0}B_{\lambda}^{-1}N_{\lambda}\geq 0$, then $\lim_{\lambda\to 0}\rho(M_{\lambda}^{-1}N_{\lambda})\leq \rho(M^{\dag}N)<1$.
\end{theorem}

\begin{proof}
Let $A^{\dag}N\geq 0$ and $\lim_{\lambda\to 0}(B_{\lambda}^{-1}N_{\lambda})\geq 0$. Then by Theorem \ref{3.1.3} and \ref{3.1.5} we obtain $\rho(M^{\dag}N)<1$ and $\lim_{\lambda\to 0}\rho(M_{\lambda}^{-1}N_{\lambda})<1$, respectively. By Theorem \ref{2.25}, the inequality $\rho(A^{\dag}N)\geq \lim_{\lambda\to 0}\rho(B_{\lambda}^{-1}N_{\lambda})$ follows from the assumption  $A^{\dag}N\geq \lim_{\lambda\to 0}(B_{\lambda}^{-1}N_{\lambda})$. Since  $\frac{\sigma}{1+\sigma}$ is a strictly increasing function in $\sigma( \geq 0)$, so we have $\frac{\rho(A^{\dag}N)}{1+\rho(A^{\dag}N)}\geq \lim_{\lambda\to 0}\frac{\rho(B_{\lambda}^{-1}N_{\lambda})}{1+\rho(B_{\lambda}^{-1}N_{\lambda})}$. In view of Theorem \ref{3.1.3} and Theorem \ref{3.1.5}, one can conclude that $\lim_{\lambda\to 0}\rho(M_{\lambda}^{-1}N_{\lambda})\leq \rho(M^{\dag}N)<1$. 
\end{proof}
In support of Theorem \ref{pp4.1conv}, the following example is worked-out.
\begin{example}\label{exthm3.2}
Let $A = \begin{bmatrix}
4 & 0 & 2\\
0 & 4 & 2\\
2 & 2 & -4\\
2 & 2 & 0
\end{bmatrix} = 
\begin{bmatrix}
160 & 80 & 60\\
120 & 160 & 60\\
80 & 80 & -4\\
80 & 80 & 0
\end{bmatrix}
-
\begin{bmatrix}
156 & 80 & 58\\
120 & 156 & 58\\
78 & 78 & 0\\
78 & 78 & 0
\end{bmatrix}=M-N$
be a weak proper splitting of the first type of $A$. Now for $\lambda = 10^{-4}$, we have 
\begin{eqnarray*}
B_{\lambda} =  \begin{bmatrix}
24.0001 & 8 & 0\\
8 & 24.0001 & 0\\
0 & 0 & 24.0001
\end{bmatrix}
&=&
\begin{bmatrix}
100 & 20 & 2\\
35 & 40 & 1\\
5 & 3 & 40
\end{bmatrix}
-
\begin{bmatrix}
75.9999 & 12 & 2\\
27 & 15.9999 & 1\\
5 & 3 & 15.9999
\end{bmatrix}\\
&=& M_{\lambda}-N_{\lambda},
\end{eqnarray*}
is a weak splitting of the first type of $B_{\lambda}$. One can easily verify that $A^{\dag}N\geq B_{\lambda}^{-1}N_{\lambda}\geq 0$ and  $0.7594 = \rho(M_{\lambda}^{-1}N_{\lambda})<\rho(M^{\dag}N)=0.9823 <1$.
\end{example}

 Note that, in the above theorem, we do not assume semi-monotone condition on $A$  as considered in \cite{labals2015} while comparing two nonnegative splittings. Similarly, we can show the next theorem for a weak splitting of the second type.
 
\begin{theorem}\label{pp4.2conv}
Let $A = M-N$ be a weak proper splitting of the second type of a singular matrix $A\in\rnn$. For $\lambda >0$, let $B_{\lambda} = M_{\lambda}-N_{\lambda}$ be a weak splitting of the second type of  $B_{\lambda}\in\rnn$. If $NA^{\dag}\geq \lim_{\lambda\to 0}N_{\lambda}B_{\lambda}^{-1}\geq 0$, then $\lim_{\lambda\to 0}\rho(M_{\lambda}^{-1}N_{\lambda})\leq \rho(M^{\dag}N)<1$.
\end{theorem}
Further, we discuss a few comparison results by considering weak splittings of alternate types.

\begin{theorem}\label{4.2conv}
Let $A = M-N$ be a weak proper splitting of the second type of a singular semi-monotone matrix $A\in\rnn$ with $NA^{\dag}\geq 0$ . For $\lambda >0$, let $B_{\lambda} = M_{\lambda}-N_{\lambda}$ be a weak splitting of the first type of the matrix $B_{\lambda}\in\rnn$ with $\lim_{\lambda\to 0}B_{\lambda}^{-1}N_{\lambda}\geq 0$. If $\lim_{\lambda\to 0}M_{\lambda}^{-1}A^{T} \geq M^{\dag}$, then $\lim_{\lambda\to 0}\rho(M_{\lambda}^{-1}N_{\lambda})\leq \rho(M^{\dag}N)<1$. 
\end{theorem}
\begin{proof}
Using Theorem \ref{3.1.3} and \ref{3.1.5}, we get $\rho(M^{\dag}N)<1$ and $\lim_{\lambda\to 0}\rho(M_{\lambda}^{-1}N_{\lambda})<1$, respectively. By Theorem \ref{2.8}, the condition $\lim_{\lambda\to 0}(M_{\lambda}^{-1}A^{T}) \geq M^{\dag}$ yields the following inequality $\lim_{\lambda\to 0}(I-M_{\lambda}^{-1}N_{\lambda})B_{\lambda}^{-1}A^{T}\geq A^{\dag}(I-NM^{\dag})$.  Applying  $A^\dagger=\lim_{\lambda\to 0} B_{\lambda}^{-1}A^{T}$  (from Theorem \ref{2.1.8}), we obtain
\begin{equation}\label{eq33.1}
    \lim_{\lambda\to 0}(B_{\lambda}^{-1}A^{T}-M_{\lambda}^{-1}N_{\lambda}B_{\lambda}^{-1}A^{T})\geq \lim_{\lambda\to 0} B_{\lambda}^{-1}A^{T}(I-NM^{\dag}).
\end{equation}
Since $M_{\lambda}^{-1}N_{\lambda}\geq 0$, by Theorem \ref{2.1} there exists a nonnegative eigenvector $x^{T}$ such that $x^{T}M_{\lambda}^{-1}N_{\lambda} = \rho(M_{\lambda}^{-1}N_{\lambda})x^{T}$. Taking limit $\lambda \to 0$ both sides, further it leads
\begin{equation}\label{eq33.2}
   \lim_{\lambda\to 0}x^{T}M_{\lambda}^{-1}N_{\lambda} = \lim_{\lambda\to 0}\rho(M_{\lambda}^{-1}N_{\lambda})x^{T}. 
\end{equation}
Pre-multiplying equation \eqref{eq33.1} by $x^{T}$, we get  \begin{equation}\label{eq33.3}
   \lim_{\lambda\to 0}x^{T}M_{\lambda}^{-1}N_{\lambda}B_{\lambda}^{-1}A^{T}\leq \lim_{\lambda\to 0}x^{T}B_{\lambda}^{-1}A^{T}NM^{\dag}. 
\end{equation}
Equation \eqref{eq33.2} and $(\ref{eq33.3})$  yields $\lim_{\lambda\to 0}\rho(M_{\lambda}^{-1}N_{\lambda})z_\lambda^{T}\leq  z_\lambda^{T}NM^{\dag}$, where $z_\lambda^{T}=\lim_{\lambda\to 0}x^{T}B_{\lambda}^{-1}A^{T}$. Taking transpose, we obtain 
\begin{equation}\label{eq33.4}
 \lim_{\lambda\to 0}\rho(M_{\lambda}^{-1}N_{\lambda})z_\lambda\leq  (NM^{\dag})^Tz_\lambda.
 \end{equation}
Now $z_\lambda^{T} = x^{T} \lim_{\lambda\to 0}B_{\lambda}^{-1}A^{T} = x^{T}A^{\dag}\geq 0$. If $z_\lambda^{T} =0$, then $ \lim_{\lambda\to 0}x^{T}B_{\lambda}^{-1}A^{T} = 0$. Further, $0 = \lim_{\lambda\to 0}A(B_{\lambda}^{-1})^{T}x = \lim_{\lambda\to 0}A^{T}A(B_{\lambda}^{-1})^{T}x = \lim_{\lambda\to 0}(A^{T}A+\lambda I)(B_{\lambda}^{-1})^{T}x = \lim_{\lambda\to 0}(A^{T}A+\lambda I)^{T}(B_{\lambda}^{-1})^{T}x = \lim_{\lambda\to 0}B_{\lambda}^{T}(B_{\lambda}^{-1})^{T}x = x$, which is a contradiction. Hence $z_\lambda^{T}> 0$. Applying Theorem \ref{2.4} to equation \eqref{eq33.4}, we conclude that 
\begin{center}
  $\lim_{\lambda\to 0}\rho(M_{\lambda}^{-1}N_{\lambda})\leq   \rho(NM^{\dag})^{T} = \rho(NM^{\dag})=\rho(M^{\dag}N)<1$.   
\end{center}
\end{proof}
The semi-monotone condition given in the Theorem \ref{4.4conv} can be relaxed   as discussed in the next result. 

\begin{theorem}\label{thm3.5}
Let $A = M-N$ be a weak proper splitting of the second type with $NA^{\dag}\geq 0$. For $\lambda >0$, suppose $B_{\lambda} = M_{\lambda}-N_{\lambda}$ is a weak splitting of the first type of the matrix $B_{\lambda}$ with $\lim_{\lambda\to 0}B_{\lambda}^{-1}N_{\lambda}\geq 0$. If $\lim_{\lambda\to 0}M_{\lambda}^{-1}N_{\lambda}A^{\dag} \leq A^{\dag}NM^{\dag}$, then $\lim_{\lambda\to 0}\rho(M_{\lambda}^{-1}N_{\lambda})\leq \rho(M^{\dag}N)<1$.
\end{theorem}
The proof will go a similar way as in Theorem \ref{4.2conv}. The converse of the Theorem \ref{thm3.5} need not true in general. One can verify it by the following example.
\begin{example}\label{exthm3.5}
Consider the matrices $A$, $M$, $N$, $B_{\lambda}$, $M_{\lambda}$, and $N_{\lambda}$ as given in Example \ref{exthm3.2}. Clearly $NA^\dagger >0$ and $A=M-N$ is a weak splitting of the second type since
\begin{center}
$NM^\dagger=\begin{bmatrix}
0.9556 &   0.0112 &   0.0013  &  0.0208\\
    0.0222  &  0.9445  &  0.0003  &  0.0385\\
    0.0108 &   0.0215  &  0.4843  &  0.4369\\
    0.0108  &  0.0215 &   0.4843  &  0.4369
\end{bmatrix}$.  Now for $\lambda = 10^{-4}$, we have $B_{\lambda}^{-1}N_{\lambda}\geq 0$.
\end{center}
 Further, from Example \ref{exthm3.2}, it follows that  $\rho(M_{\lambda}^{-1}N_{\lambda})< \rho(M^{\dag}N)<1$ but 
 \begin{center}
 $ A^{\dag}NM^{\dag}-M_{\lambda}^{-1}N_{\lambda}A^{\dag} = \begin{bmatrix}
0.0393 & -0.0168 & 0.0123 & 0.0058\\
-0.0346 & 0.1124 & 0.0384 & 0.0375\\
0.0436 & 0.0368 & -0.0187 & -0.0724
\end{bmatrix}
\ngeq 0.$
\end{center}
\end{example}

Similarly, we can show the following result for the same type of weak splittings.
\begin{lemma}\label{4.3conv}
Let $A = M-N$ be a weak proper splitting of the first type of a singular matrix  $A\in\rnn$ with $A^{\dag}N\geq 0$. For $\lambda >0$, let $B_{\lambda} = M_{\lambda}-N_{\lambda}$ be a weak splitting of the first type of the matrix $B_{\lambda}$ with $\lim_{\lambda\to 0}B_{\lambda}^{-1}N_{\lambda}\geq 0$. If $\lim_{\lambda\to 0}M_{\lambda}^{-1}N_{\lambda}A^{\dag} \leq A^{\dag}M^{\dag}N$, then $\lim_{\lambda\to 0}\rho(M_{\lambda}^{-1}N_{\lambda})\leq \rho(M^{\dag}N)<1$.
\end{lemma}

Next, we discuss another comparison theorem for different pair of weak splittings. 
\begin{theorem}\label{4.4conv}
Let $A = M-N$ be a weak proper splitting of the first type of $A\in\rmn$ with $A^{\dag}N\geq 0$. For $\lambda >0$, let $B_{\lambda} = M_{\lambda}-N_{\lambda}$ be a weak splitting of the second type of the matrix $B_{\lambda}\in\rnn$ with $\lim_{\lambda\to 0}N_{\lambda}B_{\lambda}^{-1}\geq 0$. If $\lim_{\lambda\to 0}N_{\lambda}M_{\lambda}^{-1} \leq M^{\dag}N$, then $\lim_{\lambda\to 0}\rho(M_{\lambda}^{-1}N_{\lambda})\leq \rho(M^{\dag}N)<1$.
\end{theorem}
\begin{proof}
By Theorem \ref{3.1.3} and Theorem \ref{3.1.5}, $\rho(M^{\dag}N)<1$ and $\lim_{\lambda\to 0}\rho(M_{\lambda}^{-1}N_{\lambda})<1$, respectively. The splitting of the matrix $B_{\lambda} = M_{\lambda}-N_{\lambda}$ gives  $M_{\lambda} = (I+N_{\lambda}B_{\lambda}^{-1})B_{\lambda}$. Hence $M_{\lambda}^{-1} = B_{\lambda}^{-1}(I+N_{\lambda}B_{\lambda}^{-1})^{-1}$. Applying Theorem \ref{2.6} $(ii)$ to the proper splitting $A = M-N$, we get $M^{\dag} = (I+A^{\dag}N)^{-1}A^{\dag}$. Therefore, the condition $\lim_{\lambda\to 0}N_{\lambda}M_{\lambda}^{-1} \leq M^{\dag}N$ implies 
\begin{equation}\label{coneq1}
 \lim_{\lambda\to 0}(N_{\lambda}B_{\lambda}^{-1}(I+N_{\lambda}B_{\lambda}^{-1})^{-1}) \leq (I+A^{\dag}N)^{-1}A^{\dag}N.   
\end{equation}
Since $I+A^{\dag}N \geq 0$ and $\lim_{\lambda\to 0}(I+N_{\lambda}B_{\lambda}^{-1}) \geq 0$, so pre-multiplying $I+A^{\dag}N$ and post-multiplying $\lim_{\lambda\to 0}(I+N_{\lambda}B_{\lambda}^{-1})$ to the equation (\ref{coneq1}), we obtain  
\begin{equation}\label{eq33.6}
    \lim_{\lambda\to 0}(I+A^{\dag}N)N_{\lambda}B_{\lambda}^{-1} \leq \lim_{\lambda\to 0}A^{\dag}N(I+N_{\lambda}B_{\lambda}^{-1}).
\end{equation}
 Equation \eqref{eq33.6} lead to $\lim_{\lambda\to 0}N_{\lambda}B_{\lambda}^{-1}\leq \lim_{\lambda\to 0}A^{\dag}N=A^{\dag}N$. By Theorem \ref{2.25}, we have $\rho(A^{\dag}N)\geq \lim_{\lambda\to 0}\rho( N_{\lambda}B_{\lambda}^{-1})\geq 0$. As $\frac{\gamma}{\gamma+1}$ is a strictly increasing function for every $\gamma \geq 0$, hence $\frac{\rho(A^{\dag}N)}{1+\rho(A^{\dag}N)} \geq \lim_{\lambda\to 0}\frac{\rho(N_{\lambda}B_{\lambda}^{-1})}{1+\rho(N_{\lambda}B_{\lambda}^{-1})}$. Again, by Theorem \ref{3.1.3} and Theorem \ref{3.1.5}, we get $\lim_{\lambda\to 0}\rho(M_{\lambda}^{-1}N_{\lambda})\leq \rho(M^{\dag}N)<1$.  \end{proof}

\subsection{Double weak splitting of type II}\label{sub3.2}
Motivated by the work of the authors \cite{lalic2014}, \cite{lalicx2014}, \cite{lamiao2012} and \cite{lasons1}, we have introduced the double weak splitting of type II for symmetric matrices. In connection to double weak splitting of type II, we have extended a few results of \cite{lasons1}. Further, some comparison theorems for double weak splitting of type I and type II have been established in this subsection. First, we define the double weak splitting of type II as follows.
\begin{definition}\label{w1.2}
The splitting $A = P-R+S$ is called double weak splitting of type II of  a symmetric nonsingular matrix $A\in\mathbb{R}^{n\times{n}}$ if $RP^{-1}\geq 0$ and $-SP^{-1}\geq 0$.
\end{definition}
Suppose $A = P-R+S$ be a double weak splitting of type II of a symmetric nonsingular matrix $A\in\rnn$. The iterative scheme corresponding to such type of splitting is 

\begin{equation*}\label{t2.1}
    x^{k+1} = (RP^{-1})^{T}x^{k}-(SP^{-1})^{T}x^{k-1}+(P^{-1})^{T}b,
\end{equation*}
and its block matrix representation is given by 
\begin{eqnarray}\label{t2.2}
  \nonumber  
\begin{pmatrix}
    x^{k+1} \\
    x^{k}
  \end{pmatrix}
  &=& \begin{pmatrix}
    (RP^{-1})^{T} & -(SP^{-1})^{T} \\
    I & 0
  \end{pmatrix} \begin{pmatrix}
    x^{k} \\
    x^{k-1}
  \end{pmatrix}
  +  \begin{pmatrix}
    (P^{-1})^{T}b \\
    0
  \end{pmatrix}\\
  &=& \widetilde{\mathbf{W}}\begin{pmatrix}
    x^{k} \\
    x^{k-1}
  \end{pmatrix}
  +  \begin{pmatrix}
    (P^{-1})^{T}b \\
    0
  \end{pmatrix} ,
\end{eqnarray}
where $I$ is an identity matrix of order $n$ and the iteration matrix $\widetilde{\mathbf{W}}$ is,
\begin{equation*}\label{t2.3}
     \widetilde{\mathbf{W}} = \begin{pmatrix}
    (RP^{-1})^{T} & -(SP^{-1})^{T} \\
    I & 0 
    \end{pmatrix}.
\end{equation*}

Next, we recall the following result from Song and Song \cite{lasons1}.

\begin{theorem}[Theorem 2.2, \cite{lasons1}]\label{song1.1}
Let  $A = P-R+S$ be  a double weak splitting of type I of the nonsingular matrix $A\in\rnn$. Then the double splitting is convergent if and only if $\rho(P^{-1}(R-S))<1$.
\end{theorem}
In regard to Theorem \ref{song1.1}, we have the following convergence theorem for double weak splitting of type II.

\begin{theorem}\label{weak1.1}
Let $A = P-R+S$ be a double weak splitting of type II of a symmetric nonsingular matrix $A\in\rnn$. Then $\rho(\widetilde{\mathbf{W}})<1$ if and only if $\rho((R-S)P^{-1}) = \rho(P^{-1}(R-S))<1$.
\end{theorem}

\begin{proof}
Let $A = P-R+S$ be a double weak splitting of type II. Then  $(RP^{-1})^{T}\geq 0$ and $-(SP^{-1})^{T}\geq 0$. Hence $\widetilde{\mathbf{W}}\geq 0$. By proceeding similarly lines of the proof of Theorem \ref{song1.1}, we can prove the theorem.
\end{proof}

In view of Lemma $2.7$ of \cite{lasong2002} and Theorem \ref{weak1.1}, we can  show the following result.

\begin{theorem}\label{weak1.2}
Let $A=P-R+S$ be a double weak splitting of type II of the symmetric nonsingular matrix $A\in\rnn$. Then the followings conditions are equivalent:
\begin{enumerate}
\item[(i)] $\rho(\widetilde{\mathbf{W}})<1$.
\item[(ii)] $PA^{-1}\geq 0$ $(A^{-1}P^{T}\geq 0)$.
\item[(iii)] $PA^{-1}\geq I$ $(A^{-1}P^{T}\geq I)$.
\item[(iv)] $(R-S)A^{-1}\geq 0$ $(A^{-1}(R-S)^{T}\geq 0)$.
\item[(v)] $(R-S)A^{-1}\geq -I$ $(A^{-1}(R-S)^{T}\geq -I)$.
\item[(vi)] $(I-(R-S)^{T}(P^{-1})^{T})^{-1}\geq 0$.
\item[(vii)]  $(I-(R-S)^{T}(P^{-1})^{T})^{-1}\geq I$.  
\end{enumerate}

\end{theorem}
If we consider $A = P_1-R_1+S_1 = P_2-R_2+S_2$ be two double weak splittings of a symmetric nonsingular matrix $A\in\rnn$,  Then  the respective iteration block matrices are

\begin{equation*}
     \widetilde{\mathbf{W}}_1 = \begin{pmatrix}
    (R_1P_1^{-1})^{T} & -(S_1P_1^{-1})^{T} \\
    I & 0 
    \end{pmatrix}~\mbox{and}~\widetilde{\mathbf{W}}_2 = \begin{pmatrix}
    (R_2P_2^{-1})^{T} & -(S_2P_2^{-1})^{T} \\
    I & 0 
    \end{pmatrix}.
\end{equation*}
To analyze the spectral radius of both iteration matrix $\widetilde{\mathbf{W}}_1$  and $\widetilde{\mathbf{W}}_2$, we follow the analogous of Song and Song \cite{lasons1}. We first define $\mathbb{A} = \begin{pmatrix}
    A & -I \\
    0 & I 
    \end{pmatrix}$. Then it is easy to verify $\mathbb{A}^{-1} = \begin{pmatrix}
    A^{-1} & A^{-1} \\
    0 & I 
    \end{pmatrix}$.
Further, consider $\mathbb{A}=\mathbb{M}_i-\mathbb{N}_i$
be two splitting of $\mathbb{A}\in{\mathbb{R}^{2n\times{2n}}}$. If we take \begin{equation}\label{we1.2}
    \mathbb{M}_i = \begin{pmatrix}
    P_i & 0 \\
    -S_i & I 
    \end{pmatrix}~\mbox{and}~\mathbb{N}_i = \begin{pmatrix}
    R_i-S_i & I \\
    -S_i & 0 
    \end{pmatrix},
\end{equation}
then we can show that 
\begin{equation}\label{we1.3}
\widetilde{\mathbf{W}}_i = (\mathbb{N}_i\mathbb{M}_i^{-1})^{T}, ~\mbox{for}~ i = 1,2.  
\end{equation}
We recall the comparison theorem of \cite{lasons1} which was proved for the double weak splitting of type I.
\begin{theorem}[Theorem $3.3$,\cite{lasons1}]\label{song1.3}
 Let $A = P_1-R_1+S_1 = P_2-R_2+S_2$ be two convergent and double weak splittings of type I of the monotone matrix $A\in\rnn$. If $P_1\leq P_2$ and $S_2\leq S_1$, then $\rho(\widehat{\mathbf{W}}_1)\leq \rho(\widehat{\mathbf{W}}_2)<1.$
\end{theorem}
Using similar lines of Theorem \ref{song1.3}, we can show the following result for double weak splitting type II.

\begin{theorem}\label{weak1.3}
Let $A = P_1-R_1+S_1 = P_2-R_2+S_2$ be two convergent double weak splittings of type II of a symmetric monotone matrix $A\in\rnn$. If $P_1\leq P_2$ and $S_2\leq S_1$, then $\rho(\widetilde{\mathbf{W}}_1)\leq \rho(\widetilde{\mathbf{W}}_2)<1$.
\end{theorem}

Corollary $3.13$ of \cite{lagiri2017} motivated us to study the above  comparison theorem without considering the monotone condition and established the following result.

\begin{theorem}\label{weak1.901}
Let $A = P_1-R_1+S_1 = P_2-R_2+S_2$ be two double weak splitting of type II of the  symmetric matrix $A\in\rnn$. If $P_2A^{-1}\geq P_1A^{-1}\geq 0$ and $S_2A^{-1}\leq S_1A^{-1}\leq 0$, then $\rho(\widetilde{\mathbf{W}}_1)\leq \rho(\widetilde{\mathbf{W}}_2)<1$.
\end{theorem}
\begin{proof}
From the conditions $P_1A^{-1}\geq 0$, $P_2A^{-1}\geq 0$ and Theorem \ref{weak1.2}, it is trivial that $\rho(\widetilde{\mathbf{W}}_1)<1$ and $\rho(\widetilde{\mathbf{W}}_2)<1$. Since $P_2A^{-1}\geq P_1A^{-1}\geq 0$ and $S_2A^{-1}\leq S_1A^{-1}\leq 0$, so we have 
\begin{equation*}
    \begin{pmatrix}
    P_2A^{-1} &  P_2A^{-1} \\
    -S_2A^{-1} & -S_2A^{-1}+I
    \end{pmatrix}\geq \begin{pmatrix}
    P_1A^{-1} &  P_1A^{-1} \\
    -S_1A^{-1} & -S_1A^{-1}+I
    \end{pmatrix}\geq 0.
\end{equation*}
Further, we can write as  
\begin{equation*}
    \begin{pmatrix}
    P_2 & 0 \\
    -S_2 & I
    \end{pmatrix}\begin{pmatrix}
    A^{-1} & A^{-1} \\
    0 & I
    \end{pmatrix}\geq \begin{pmatrix}
    P_1 & 0 \\
    -S_1 & I
    \end{pmatrix}\begin{pmatrix}
    A^{-1} & A^{-1} \\
    0 & I
    \end{pmatrix} \geq 0.
\end{equation*}
Hence by equation $(\ref{we1.2})$, $\mathbb{M}_2\mathbb{A}^{-1}\geq \mathbb{M}_1\mathbb{A}^{-1}\geq 0$. Taking $j=1$ and $\alpha=1$ in Theorem \ref{giricor},  for the splittings $\mathbb{A} = \mathbb{M}_1-\mathbb{N}_1 = \mathbb{M}_2-\mathbb{N}_2$, we can show $\rho(\mathbb{N}_1\mathbb{M}_1^{-1})\leq \rho(\mathbb{N}_2\mathbb{M}_2^{-1})<1$. Thus by equation (\ref{we1.3}),  $\rho(\widetilde{\mathbf{W}}_1)\leq \rho(\widetilde{\mathbf{W}}_2)<1$.
\end{proof}
In support of Theorem \ref{weak1.901}, the following example is worked-out.

\begin{example}\label{exthm3.13}
Consider \begin{eqnarray*}
A = \begin{bmatrix}
10 & -4\\
-4 & 6
\end{bmatrix} 
&=&
\begin{bmatrix}
12 & 0\\
0 & 8
\end{bmatrix}
-
\begin{bmatrix}
2 & 2\\
4 & 2
\end{bmatrix}
+
\begin{bmatrix}
0 & -2\\
0 & 0
\end{bmatrix}= P_1-R_1+S_1\\
&=& \begin{bmatrix}
16 & 0\\
0 & 10
\end{bmatrix}-\begin{bmatrix}
6 & 2\\
0 & 4
\end{bmatrix}
+
\begin{bmatrix}
0 & -2\\
-4 & 0
\end{bmatrix}= P_2-R_2+S_2
\end{eqnarray*}
be two convergent double weak splitting of type II 
of the matrix $A$. One can verify that 
\begin{center}
    $ P_2A^{-1}=\begin{bmatrix}
   2.1818 &   1.4545\\
    0.9091  &  2.2727
    \end{bmatrix}>\begin{bmatrix}
    1.6364  &  1.0909\\
    0.7273  &  1.8182
    \end{bmatrix}=P_1A^{-1}>0$, 
\end{center}
\begin{center}
    $ S_2A^{-1}=\begin{bmatrix}
   -0.1818  & -0.4545\\
   -0.5455  & -0.3636
\end{bmatrix}\leq \begin{bmatrix}
   -0.1818  & -0.4545\\
         0   &      0
    \end{bmatrix}=S_1A^{-1}\leq 0$, and
    \end{center}
$0.6667 = \rho(\widetilde{\mathbf{W}}_1)<\rho(\widetilde{\mathbf{W}}_2)=0.7729 <1$. 
\end{example}

Another comparison theorem for symmetric  nonsingular matrices presented below.
\begin{theorem}\label{weak1.90}
Let $A = P_1-R_1+S_1 = P_2-R_2+S_2$ be two convergent double weak splitting of type II of a symmetric matrix $A\in\mathbb{R}^{n\times{n}}$. If $R_1P_1^{-1}\geq R_2P_2^{-1}$ and $AP_1^{-1}\geq AP_2^{-1}$, then $\rho(\widetilde{\mathbf{W}}_1)\leq \rho(\widetilde{\mathbf{W}}_2)<1$.
\end{theorem}

\begin{proof}
If $\rho(\widetilde{\mathbf{W}}_1) = 0$, then it is trivial. Assume that $0<\rho(\widetilde{\mathbf{W}}_1)<1$. Since $\widetilde{\mathbf{W}}_1\geq 0$, so by Theorem \ref{2.1}, there exist a  eigenvector $x=(x_1,x_2)^T$ such that $\widetilde{\mathbf{W}}_1x = \rho(\widetilde{\mathbf{W}}_1)x$. Which implies 
\begin{equation}\label{eq33.10}
    (R_1P_1^{-1})^{T}x_1-(S_1P_1^{-1})^{T}x_2 = \rho(\widetilde{\mathbf{W}}_1)x_1 \mbox{ and }
    x_1 = \rho(\widetilde{\mathbf{W}}_1)x_2.
\end{equation}
Now
\begin{eqnarray}\label{eq33.11}
\nonumber
&&(R_2P_2^{-1})^{T}x_1-(S_2P_2^{-1})^{T}x_2-\rho(\widetilde{\mathbf{W}}_1)x_1\\
\nonumber
&& \hspace{2cm}=(R_2P_2^{-1})^{T}x_1-\frac{1}{\rho(\widetilde{\mathbf{W}}_1)}(S_2P_2^{-1})^{T}x_1-(R_1P_1^{-1})^{T}x_1+\frac{1}{\rho(\widetilde{\mathbf{W}}_1)}(S_1P_1^{-1})^{T}x_1\\
\nonumber
&& \hspace{2cm} \geq \frac{1}{\rho(\widetilde{\mathbf{W}}_1)}[(R_2P_2^{-1})^{T}-(R_1P_1^{-1})^{T}+(S_1P_1^{-1})^{T}-(S_2P_2^{-1})^{T}]x_1\\
\nonumber
&& \hspace{2cm}=\frac{1}{\rho(\widetilde{\mathbf{W}}_1)}[(P_2^{-1})^{T}(R_2^{T}-S_2^{T})+(P_1^{-1})^{T}(S_1^{T}-R_1^{T})]x_1\\
\nonumber
&& \hspace{2cm}=\frac{1}{\rho(\widetilde{\mathbf{W}}_1)}[(P_2^{-1})^{T}(P_2^{T}-A)+(P_1^{-1})^{T}(A-P_1^{T})]x_1\\
&& \hspace{2cm}= \frac{1}{\rho(\widetilde{\mathbf{W}}_1)}[(P_1^{-1})^{T}A-(P_2^{-1})^{T}A]x_1\geq 0.
\end{eqnarray}
Using equations \eqref{eq33.10} and \eqref{eq33.11}, we obtain 
\begin{center}
$\widetilde{\mathbf{W}}_2x-\rho(\widetilde{\mathbf{W}}_1)x = \begin{pmatrix}
    (R_2P_2^{-1})^{T}x_1-(S_2P_2^{-1})^{T}x_2-\rho(\widetilde{\mathbf{W}}_1)x_1\\ 
    x_1-\rho(\widetilde{\mathbf{W}}_1)x_2
    \end{pmatrix}\geq 0$.
\end{center}
Hence by Theorem \ref{2.4}, $\rho(\widetilde{\mathbf{W}}_1)\leq \rho(\widetilde{\mathbf{W}}_2)<1$. 
\end{proof}
The converse of the Theorem \ref{weak1.90} need not true in general as shown in the below example. 
\begin{example}\label{conv1.2}
Consider the matrices $A$, $P_1$, $R_1$, $S_1$, $P_2$, $R_2$, and $S_2$ as given in Example \ref{exthm3.13}. Clearly $0.6667 = \rho(\widetilde{\mathbf{W}}_1)< \rho(\widetilde{\mathbf{W}}_2)=0.7729 <1$ but
\begin{equation*}
    R_1P_1^{-1}-R_2P_2^{-1}= \begin{bmatrix}
-0.2083 & 0.0500\\
 0.3333 & -0.1500
\end{bmatrix}
\ngeq 0, AP_1^{-1}-AP_2^{-1}= \begin{bmatrix}
 0.2083  & -0.1000\\
   -0.0833  &  0.1500
\end{bmatrix}
\ngeq 0
\end{equation*}
\end{example}

On account of Theorem $3.6$ \cite{lasons1} and equations (\ref{we1.2}) and (\ref{we1.3}), the following comparison theorem similarly follows. 

\begin{theorem}\label{weak1.4}
Let $A = P_1-R_1+S_1 = P_2-R_2+S_2$ be two convergent double weak splittings of type II of the nonsingular symmetric matrix $A\in\rnn$. If any one of the following conditions
\begin{enumerate}
    \item[(i)] $P_2P_1^{-1}\geq I$ and $S_1P_1^{-1}\geq S_2P_1^{-1}$,
        \item[(ii)] $P_1P_2^{-1}\leq I$ and $S_1P_2^{-1}\geq S_2P_2^{-1}$,
\end{enumerate}
holds, then $\rho(\widetilde{\mathbf{W}}_1)\leq \rho(\widetilde{\mathbf{W}}_2)<1$.
\end{theorem}

Analogous to the Theorem $3.1$ of \cite{lamio2009}, Corollary $4.10$ of \cite{ladm1}, and the proof of Theorem \ref{weak1.90},  we obtain the below result for double weak splittings of type II.

\begin{theorem}\label{weak1.9}
Let $A_1 = P_1-R_1+S_1$ and $A_2  = P_2-R_2+S_2$ be two convergent double weak splitting of type II of the nonsingular symmetric matrices $A_1\in\mathbb{R}^{n\times{n}}$ and $A_2\in\mathbb{R}^{n\times{n}}$, respectively. If $R_1P_1^{-1}\geq R_2P_2^{-1}$ and $A_1P_1^{-1}\geq A_2P_2^{-1}$, then $\rho(\widetilde{\mathbf{W}}_1)\leq \rho(\widetilde{\mathbf{W}}_2)<1$.
\end{theorem}

Similarly, the following results can be proved by considering a double weak splitting of type I for the matrices $A_1$ and $A_2$. 
\begin{theorem}\label{weak1.10}
Let  $A_1 = P_1-R_1+S_1$ be a convergent double weak splitting of type II of the nonsingular symmetric matrix $A_1\in\mathbb{R}^{n\times{n}}$ and $A_2  = P_2-R_2+S_2$ be a convergent double weak splitting of type I of the nonsingular symmetric matrix  $A_2\in\mathbb{R}^{n\times{n}}$. If $P_2^{-1}R_2\leq (R_1P_1^{-1})^{T}$ and $ P_2^{-1}A_2\leq (P_1^{-1})^{T}A_1$, then $\rho(\widetilde{\mathbf{W}}_1)\leq \rho(\widehat{\mathbf{W}}_2)<1$.
\end{theorem}

Exchanging the splitting type in Theorem \ref{weak1.10}, we obtain the below result.

\begin{theorem}\label{weak1.11}
Let  $A_1 = P_1-R_1+S_1$ be a convergent double weak splitting of type I of the nonsingular symmetric matrix $A_1\in\mathbb{R}^{n\times{n}}$ and $A_2  = P_2-R_2+S_2$ be a convergent double weak splitting of type II of the nonsingular symmetric matrix  $A_2\in\mathbb{R}^{n\times{n}}$. If $(R_2P_2^{-1})^{T}\leq P_1^{-1}R_1$ and $(P_2^{-1})^{T}A_2\leq P_1^{-1}A_1$, then $\rho(\widehat{\mathbf{W}}_1)\leq \rho(\widetilde{\mathbf{W}}_2)<1$.
\end{theorem}

\subsection{Convergence and Comparison using double weak splittings}\label{sub3.}
In this subsection, we introduce a regularized iterative scheme for the well-posed system (\ref{lareq1}) based on double weak splittings. In addition, the convergence of the regularized scheme is established. Further, a few comparison theorems for the systems (\ref{laeq1}) \& (\ref{lareq1}) are analyzed with the help of double weak splittings.
 
Let $B_{\lambda} = P_\lambda-R_\lambda+S_\lambda$ be  a double splitting (introduced by Wo{\'z}nicki \cite{lawoz1}) of the nonsingular matrix $B_{\lambda}\in\rnn$. If $P_\lambda$ is invertible, then  regularized iterative scheme corresponding to the double splitting $B_{\lambda} = P_\lambda-R_\lambda+S_\lambda$  is given by 
\begin{center}
$x^{k+1} = P_\lambda^{-1}R_\lambda x^{k}-P_\lambda^{-1}S_\lambda x^{k-1}+P_\lambda^{-1}A^{T}b$.  
\end{center}
Further, its block matrix form is 
\begin{equation}\label{eq5}
\begin{pmatrix}
    x^{k+1} \\
    x^{k}
  \end{pmatrix}
  = \begin{pmatrix}
    P_\lambda^{-1}R_\lambda & -P_\lambda^{-1}S_\lambda \\
    I & 0
  \end{pmatrix} \begin{pmatrix}
    x^{k} \\
    x^{k-1}
  \end{pmatrix}
  +  \begin{pmatrix}
    P_\lambda^{-1}A^{T}b \\
    0
  \end{pmatrix}=\mathbf{W}_{\lambda}\begin{pmatrix}
    x^{k} \\
    x^{k-1}
  \end{pmatrix}
  +  \begin{pmatrix}
    P_\lambda^{-1}A^{T}b \\
    0
  \end{pmatrix},
\end{equation}
where $I$ is an identity matrix of order $n$ and $\mathbf{W}_{\lambda} = \begin{pmatrix}
    P_\lambda^{-1}R_\lambda & -P_\lambda^{-1}S_\lambda \\
    I & 0   
\end{pmatrix}$ is the iteration matrix.

Ensuing the idea of Golub et al. \cite{lagolv1} and \cite{lasheh1},  the iterative scheme (\ref{eq5}) for the system $B_{\lambda}x = A^{T}b$, converges to $B_{\lambda}^{-1}A^{T}b$ ($=A^{\dag}b$ as $\lambda \to 0$) for any initial vectors $x^{0}$ and $x^{1}$ if and only if $\lim_{ \lambda \to 0 }\rho(\mathbf{W}_{\lambda})<1$.  

In case of convergent double weak splittings of type I or type II and by virtue of Theorem \ref{2.1.10}, the convergence of (\ref{lareq1}) follows easily and presented below.
\begin{theorem}\label{4.1convd}
For any matrix $A\in\rmn$ and $\lambda >0$, let $B_{\lambda} = P_\lambda-R_\lambda+S_\lambda$ be a  double weak splitting of type I (respectively, type II) with $\lim_{\lambda\to 0}B_{\lambda}^{-1}P_\lambda\geq 0$ (respectively, $\lim_{\lambda\to 0}P_{\lambda}B_{\lambda}^{-1}\geq 0$), then the iterative scheme $(\ref{eq5})$ converges to $B_{\lambda}^{-1}A^{T}b$ ($ = A^{\dag}b$ as $\lambda \to 0 $).
\end{theorem}

Under the suitable sufficient condition, the following theorem signifies that the splitting of $B_{\lambda}$ will converge faster (in terms of spectral radius) than the splitting of the original matrix $A$.

\begin{theorem}\label{d2.4}
Let $A = P-R+S$ be a double proper weak splitting of type I of $A\in\rmn$ with $A^{\dag}P\geq 0$. For $\lambda >0$, let $B_{\lambda} = P_\lambda-R_\lambda+S_\lambda$ be a double weak splitting of type I with $B_{\lambda}^{-1}P_\lambda\geq 0$. If any one of the following conditions holds
\begin{enumerate}
    \item[(i)] $P^{\dag}R\geq \lim_{\lambda\to 0} P_\lambda^{-1}R_\lambda$ and $\lim_{\lambda\to 0} P_\lambda^{-1}S_\lambda \geq P^{\dag}S$,
    \item[(ii)] $P^{\dag}(R-S)\geq I$,
\end{enumerate}
then $\lim_{\lambda\to 0}\rho(\mathbf{W}_{\lambda})\leq \rho(\mathbf{W})<1$.  
\end{theorem}

\begin{proof}
By Theorem \ref{weaktp1.2} and Theorem \ref{2.1.10}, it is clear that $\lim_{\lambda\to 0} \rho(\mathbf{W}_{\lambda})<1$ and $\rho(\mathbf{W})<1$. If $\lim_{\lambda\to 0}\rho(\mathbf{W}_{\lambda}) = 0$, then the theorem is trivial. Let us assume that $\lim_{\lambda\to 0}<\lim_{\lambda\to 0}\rho(\mathbf{W}_{\lambda})<1$. Since $\lim_{\lambda\to 0}\mathbf{W}_{\lambda}> 0$, by Theorem \ref{2.1} there exist a  vector $x (\neq 0)\in\mathbb{R}^2$ such that
$\lim_{\lambda\to 0}\mathbf{W}_{\lambda}x = \lim_{\lambda\to 0}\rho(\mathbf{W}_{\lambda})x$. This implies
\begin{center}
$\lim_{\lambda\to 0}P_\lambda^{-1}R_\lambda x_1-\lim_{\lambda\to 0}P_\lambda^{-1}S_\lambda x_2= \lim_{\lambda\to 0}\rho(\mathbf{W}_{\lambda})x_1$ and $
x_1 = \lim_{\lambda\to 0}\rho(\mathbf{W}_{\lambda})x_2$.
\end{center}
Now 
\begin{center}
$\mathbf{W}x-\lim_{\lambda\to 0}\rho(\mathbf{W}_{\lambda})x = \begin{pmatrix}
    P^{\dag}Rx_1-P^{\dag}Sx_2-\lim_{\lambda\to 0}\rho(\mathbf{W}_{\lambda})x_1\\ 
    x_1-\lim_{\lambda\to 0}\rho(\mathbf{W}_{\lambda})x_2
    \end{pmatrix}= \begin{pmatrix}
   \Delta_1\\ 
    0
    \end{pmatrix}$,
\end{center}
where $\Delta_1=P^{\dag}Rx_1-P^{\dag}Sx_2-\lim_{\lambda\to 0}\rho(\mathbf{W}_{\lambda})x_1$.\\
Case-I: Let $P^{\dag}R\geq \lim_{\lambda\to 0}P_\lambda^{-1}R_\lambda$ and $\lim_{\lambda\to 0}P_\lambda^{-1}S_\lambda \geq P^{\dag}S$. Then
\begin{eqnarray*}
\Delta_1 &=& P^{\dag}Rx_1-P^{\dag}Sx_2-\lim_{\lambda\to 0}\rho(\mathbf{W}_{\lambda})x_1\\
&=& P^{\dag}Rx_1-\lim_{\lambda\to 0}\frac{1}{\rho(\mathbf{W}_{\lambda})}P^{\dag}Sx_1-\lim_{\lambda\to 0}P_\lambda^{-1}R_\lambda x_1+\lim_{\lambda\to 0}\frac{1}{\rho(\mathbf{W}_{\lambda})}P_\lambda^{-1}S_\lambda x_1\\
&=& [(P^{\dag}R-\lim_{\lambda\to 0}P_\lambda^{-1}R_\lambda) +\lim_{\lambda\to 0}\frac{1}{\rho(\mathbf{W}_{\lambda})}(P_\lambda^{-1}S_\lambda-P^{\dag}S)]x_1\\
&\geq& 0.
\end{eqnarray*}
Hence $\mathbf{W}x-\lim_{\lambda\to 0}\rho(\mathbf{W}_{\lambda})x\geq 0$.\\
Case-II: Let $P^{\dag}(R-S)\geq I$. Then
\begin{eqnarray*}
\Delta_1 &=& P^{\dag}Rx_1-P^{\dag}Sx_2-\lim_{\lambda\to 0}\rho(\mathbf{W}_{\lambda})x_1\\
&=& \lim_{\lambda\to 0}\rho(\mathbf{W}_{\lambda})P^{\dag}Rx_2-P^{\dag}Sx_2-\lim_{\lambda\to 0}\rho(\mathbf{W}_{\lambda})^{2}x_2\\
&\geq & \lim_{\lambda\to 0}\rho(\mathbf{W}_{\lambda})^{2}P^{\dag}Rx_2-\lim_{\lambda\to 0}\rho(\mathbf{W}_{\lambda})^{2}P^{\dag}Sx_2-\lim_{\lambda\to 0}\rho(\mathbf{W}_{\lambda})^{2}x_2\\
&=& \lim_{\lambda\to 0}\rho(\mathbf{W}_{\lambda})^{2}[P^{\dag}(R-S)-I]x_2\\
&\geq & 0
\end{eqnarray*}
In both cases, we obtain $\mathbf{W}x-\lim_{\lambda\to 0}\rho(\mathbf{W}_{\lambda})x \geq 0$.  Thus by Theorem \ref{2.4}, $\lim_{\lambda\to 0}\rho(\mathbf{W}_{\lambda})\leq \rho(\mathbf{W})<1$.
\end{proof}
One numerical example is given below to demonstrate Theorem \ref{d2.4}.
\begin{example}\label{pe1.2.1}
Let us consider 
\begin{align*}
A = \begin{bmatrix}
4 & 2 & 0\\
2 & 1 & 0\\
0 & 0 & 2
\end{bmatrix}&=
\begin{bmatrix}
12 & 6 & 0\\
4 & 2 & 0\\
0 & 0 & 3
\end{bmatrix}
-
\begin{bmatrix}
7 & 3 & 0\\
3 & 0.5 & 0\\
0 & 0 & 0.5
\end{bmatrix}
+
\begin{bmatrix}
-1 & -1 & 0\\
-1 & -0.5 & 0\\
0 & 0 & -0.5
\end{bmatrix}\\
&= P-R+S.
\end{align*}
We can verify that $A=P-R+S$ is a double proper weak splitting of the type I with $A^{\dag}P\geq 0$. For $\lambda = 10^{-4}$, we have 
\begin{eqnarray*}
B_{\lambda}&=&  \begin{bmatrix}
20.0001 & 10 & 0\\
10 & 5.0001 & 0\\
0 & 0 & 4.0001
\end{bmatrix}= P_{\lambda}-R_{\lambda}+S_{\lambda}\\
&=&
\begin{bmatrix}
24.0001 & 12 & 0\\
12 & 6.0001 & 0\\
0 & 0 & 4.5001
\end{bmatrix}
-
\begin{bmatrix}
2 & 0 & 0\\
1 & 0 & 0\\
0 & 0 & 0.5
\end{bmatrix}
+
\begin{bmatrix}
-2 & -2 & 0\\
-1 & -1 & 0\\
0 & 0 & 0
\end{bmatrix}
\end{eqnarray*}
which  is a convergent double weak splitting of type I with $B_{\lambda}^{-1}P_{\lambda}\geq 0$. Further, 

\begin{equation*}
    P^{\dag}R-P_\lambda^{-1}R_\lambda = \begin{bmatrix}
0.4133 & 0.1900 & 0\\
0.2067 & 0.0950 & 0\\
0 & 0 & 0.0556 
\end{bmatrix} \geq 0,
\end{equation*}
\begin{equation*}
P_\lambda^{-1}S_\lambda-P^{\dag}S  = \begin{bmatrix}
0.0133 & 0.0033 & 0\\
0.0067 & 0.0017 & 0\\
0 & 0 & 0.1667 
\end{bmatrix} \geq 0,
\end{equation*}
and $0.35192 = \rho(\mathbf{W}_{\lambda})<\rho(\mathbf{W}) = 0.7321<1$.
\end{example}

In the next result, we discuss a comparison theorem for  considering a double weak splitting of type II for the regularized matrix $B_{\lambda}$.

\begin{theorem}\label{lam1.1}
Let $A = P-R+S$ be a convergent double proper weak splitting of type I of the singular symmetric matrix $A\in\rnn$. For $\lambda >0$, let $B_{\lambda}  = P_{\lambda}-R_{\lambda}+S_{\lambda}$ be a convergent double weak splitting of type II of the nonsingular matrix $B_{\lambda}$. If $\lim_{\lambda\to 0}(R_{\lambda}P_{\lambda}^{-1})^{T}\geq P^{\dag}R$, $\lim_{\lambda\to 0}(P_{\lambda}^{-1})^{T}B_{\lambda}^{T}\geq P^{\dag}A$, $P^{\dag}R>0$ and $-P^{\dag}S>0$, then $\lim_{\lambda\to 0}\rho(\widetilde{\mathbf{W}}_{\lambda})\leq \rho(\mathbf{W})<1$.
\end{theorem}

\begin{proof}
If $\rho(\mathbf{W}) = 0$, then it is trivial. Assume that $0<\rho(\mathbf{W})<1$. From $P^{\dag}R>0$ and $-P^{\dag}S>0$ and  Lemma $3.1$ of \cite{lasheh1}, $\mathbf{W}$ is irreducible. By Theorem \ref{la2.10} there exists a positive vector $x =(x_1, x_2^T\in\mathbb{R}^2$ such that $\mathbf{W}x = \rho(\mathbf{W})x$. This leads to
\begin{equation}\label{eq33.13}
 P^{\dag}Rx_1-P^{\dag}S x_{2}=\rho(\mathbf{W})x_1 \mbox{ and }x_1= \rho(\mathbf{W})x_2.
\end{equation}
From equation \eqref{eq33.13}, it is clear that $x_1\in R(P^{\dag}) = R(P^{T})$. Now the iteration matrix corresponding to the double weak splitting of type II ($B_{\lambda}  = P_{\lambda}-R_{\lambda}+S_{\lambda}$ ) is  
\begin{equation*}
 \widetilde{\mathbf{W}}_{\lambda} = \begin{pmatrix}
    (R_\lambda P_\lambda^{-1})^{T} & -(S_\lambda P_\lambda^{-1})^{T} \\
    I & 0   
\end{pmatrix}.
\end{equation*}
Using $\widetilde{\mathbf{W}}_{\lambda}$ and equation\eqref{eq33.13}, we have 
\begin{eqnarray*}
\lim_{\lambda\to 0}\widetilde{\mathbf{W}}_{\lambda}x-\rho(\mathbf{W})x &=& \begin{pmatrix}
   \lim_{\lambda\to 0}(R_\lambda P_\lambda^{-1})^{T}x_1-\lim_{\lambda\to 0}(S_\lambda P_\lambda^{-1})^{T}x_2-\rho(\mathbf{W})x_1\\ 
    x_1-\rho(\mathbf{W})x_2
    \end{pmatrix}\\  
&=& \begin{pmatrix}
   (\lim_{\lambda\to 0}(R_\lambda P_\lambda^{-1})^{T}-P^{\dag}R)x_1-\frac{1}{\rho(\mathbf{W})} (\lim_{\lambda\to 0}(S_\lambda P_\lambda^{-1})^{T}-P^{\dag}S)x_1\\ 
   0
    \end{pmatrix}.
    \end{eqnarray*}
  Applying the condition $\lim_{\lambda\to 0}(R_{\lambda}P_{\lambda}^{-1})^{T}\geq P^{\dag}R$ and $P^{\dag}A-\lim_{\lambda\to 0}(P_{\lambda}^{-1})^{T}B_{\lambda}^{T}\leq 0$ , we obtain 
  \begin{eqnarray*}
  \lim_{\lambda\to 0}\widetilde{\mathbf{W}}_{\lambda}x-\rho(\mathbf{W})x&\leq& \frac{1}{\rho(\mathbf{W})}(\lim_{\lambda\to 0}(R_\lambda P_\lambda^{-1})^{T}-P^{\dag}R)x_1-\frac{1}{\rho(\mathbf{W})} (\lim_{\lambda\to 0}(S_\lambda P_\lambda^{-1})^{T}-P^{\dag}S)x_1\\
  &=& \frac{1}{\rho(\mathbf{W})}\lim_{\lambda\to 0}[(P_{\lambda}^{-1})^{T}(P_{\lambda}^{T}-B_{\lambda}^{T})+P^{\dag}(A-P)]x_1\\
  &=& \frac{1}{\rho(\mathbf{W})}(x_1-\lim_{\lambda\to 0}(P_{\lambda}^{-1})^{T}B_{\lambda}^{T}x_1+P^{\dag}Ax_1-x_1)
  \\
  &=& \frac{1}{\rho(\mathbf{W})}(P^{\dag}A-\lim_{\lambda\to 0}(P_{\lambda}^{-1})^{T}B_{\lambda}^{T})x_1\leq 0.
  \end{eqnarray*}
 Hence by Theorem \ref{2.4}, we get $\lim_{\lambda\to 0}\rho(\widetilde{\mathbf{W}}_{\lambda})\leq \rho(\mathbf{W})<1$. 
\end{proof}

\section{Conclusion}\label{con}
The notion of double weak splitting of type II was introduced along with regularized iterative scheme via Tikhonov's regularization. In regards to the double weak splitting of type II, a few convergence and comparison theorems have been proved.   The results in Section \ref{main} show that if we consider the regularized iterative scheme based on the splitting of $B_{\lambda}$ with some prescribed conditions, it converges faster (in terms of spectral radius) than the iterative scheme generated by the splitting of the original matrix $A$. Further, some comparison results are established with the help of weak splitting of the first type and second type, where we do not assume the monotone condition. Several equivalent comparison theorems of various combinations of weak splittings are also demonstrated. 

For future research perspectives, it is interesting to study the following points.
\begin{enumerate}
    \item The results derived in subsection \ref{sub3.2} can be extended to singular symmetric matrices.
    \item The three-step alternating iterative schemes derived in \cite{lanandi,lanandi1}, confirms that further extension can be possible by considering the alternating regularized iterative scheme. 
    \item The same idea can be developed for P-proper splittings \cite{laraj}.
    \item As tensors are natural extensions to matrices, one possible research could be to consider the multilinear system of tensor equations.
\end{enumerate}

%\section{Acknowledgements}

% The authors would like to express their sincere thanks and gratitude to the editors and the anonymous referees for their valuable comments and suggestions in the improvement of the
% manuscript.

\end{document}